\documentclass[12pt]{extarticle}
\usepackage{amsmath, amsthm, amssymb, color}
\usepackage[colorlinks=true,linkcolor=blue,urlcolor=blue]{hyperref}
\usepackage{graphicx}
\usepackage{caption}
\usepackage{mathtools}
\usepackage{enumerate}
\usepackage{verbatim}
\usepackage{tikz,tikz-cd,tikz-3dplot}
\usepackage{amssymb}
\usetikzlibrary{matrix}
\usetikzlibrary{arrows}
\usepackage{algorithm}
\usepackage[noend]{algpseudocode}
\usepackage{caption}
\usepackage[normalem]{ulem}
\usepackage{subcaption}
\oddsidemargin \evensidemargin
\usepackage{makecell}
\usepackage{array}
\usepackage{enumitem}
\newtheorem{theorem}{Theorem}[section]
\newtheorem{proposition}[theorem]{Proposition}
\newtheorem{lemma}[theorem]{Lemma}

\theoremstyle{definition}
\newtheorem{definition}[theorem]{Definition}

\newtheorem{remark}[theorem]{Remark}
\newtheorem{conjecture}[theorem]{Conjecture}

\newtheorem{examplex}[theorem]{Example}
\newenvironment{example}
{\pushQED{\qed}\examplex}
{\popQED\endexamplex}

\newtheorem{question}[theorem]{Question}

\makeatletter
\def\Ddots{\mathinner{\mkern1mu\raise\p@
\vbox{\kern7\p@\hbox{.}}\mkern2mu
\raise4\p@\hbox{.}\mkern2mu\raise7\p@\hbox{.}\mkern1mu}}
\makeatother

\title{Positive Polytopes with Few Facets in the Grassmannian}

\author{Dmitrii Pavlov and Kristian Ranestad}
\date{}
\begin{document}
\maketitle

\begin{abstract}
    In this article we study adjoint hypersurfaces of geometric objects obtained by intersecting simple polytopes with few facets in $\mathbb{P}^5$ with the Grassmannian $\mathrm{Gr}(2,4)$. These generalize the positive Grassmannian, which is the intersection of $\mathrm{Gr}(2,4)$ with the simplex. We show that if the resulting object has five facets, it is a positive geometry and the adjoint hypersurface is unique. For the case of six facets we show that the adjoint hypersurface is not necessarily unique and give an upper bound on the dimension of the family of adjoints. We illustrate our results with a range of examples. In particular, we show that even if the adjoint is not unique, a positive hexahedron can still be a positive geometry. 
\end{abstract}

\section{Introduction}

Positive geometries, which are geometric objects introduced in \cite{arkani2017positive} in the context of particle physics, have received substantial attention from the mathematical community in recent years. Loosely speaking, a positive geometry is a pair of an algebraic variety~$X$ and a semialgebraic set~$X_{\geq 0}$ inside the real points of $X$, equipped with a distinguished rational top-form on $X$, called the \emph{canonical form}, that captures the boundary structure of $X_{\geq 0}$ (for a precise definition, see \cite[Definition 1]{Lam:2022yly}). Because of a number of technical conditions imposed on the canonical form, finding examples of positive geometries is a fairly hard task. 

Standard examples of positive geometries include polytopes in the projective space $\mathbb{P}^n$, the positive Grassmannian $\mathrm{Gr}_{\geq 0}(k,n)$ inside $\mathrm{Gr}(k,n)$, positive parts of toric varieties \cite[Section 5]{arkani2017positive}, and the positive part $(\mathcal{M}_{0,n})_{\geq 0}$ of the Deligne-Knudsen-Mumford compactification $\overline{\mathcal{M}_{0,n}}$ of the moduli space of rational curves with $n$ marked points \cite[Proposition 8.2]{arkani2021cluster}. All these objects have one thing in common: a remarkably rich underlying combinatorial structure. In particular, all of them allow for a stratification of their algebraic boundary (e.g. the face stratification for polytopes and the positroid stratification for $\mathrm{Gr}_{\geq 0}(k,n)$ \cite{postnikov}). The fact that their boundaries admit such a stratification allows to determine whether these objects are positive geometries and find their canonical forms by using a combinatorial approach, namely by studying their \emph{residual arrangements}. While we refer the reader to Section \ref{sec:2} for precise definitions, we will now briefly explain the idea of this approach. For a set whose boundary has a well defined stratification consisting of faces contained in the intersection of a finite set of facets, there is a residual arrangement.  It consists of all irreducible components of intersections of (Zariski closures of) facets that are not Zariski closures of a face. 
The canonical form has simple poles along each facet, so the denominator of the rational function defining the canonical form is the product of the defining equations of the facets. The canonical form is only allowed to have poles along the faces, and therefore the numerator of the rational function must cancel the zeros of the denominator located away from the faces of our set. These zeros are precisely in the residual arrangement. Thus, the numerator defines a hypersurface of minimal possible degree interpolating the residual arrangement. Hypersurfaces with this property are called \emph{adjoint}. This classical algebraic geometric term denotes a divisor whose strict transform on the blow up along the residual arrangement restricts to a canonical divisor on the strict transform of the algebraic boundary.
The uniqueness of the canonical form may be achieved by the uniqueness of such a canonical divisor, i.e. by the uniqueness of the adjoint.

Studying adjoint hypersurfaces has proven to be rather helpful in showing that a given object is a positive geometry. An illustration of this is the case of planar polypols, nonlinear analogs of polygons, studied in \cite{kohn2021adjoints}. To introduce an example that motivated the development of the positive geometries machinery, we now take a small detour into physics. 

One of the main problems in theoretical particle physics is that of calculating \emph{scattering amplitudes}. These are quantities encoding probabilities of certain particle interactions in a given quantum field theory. One of the main recent observations in this area is that these quantities can be read off of a mathematical object called \emph{the amplituhedron} \cite{arkani2014amplituhedron}. The tree $(n,k,m)$-amplituhedron, corresponding to tree-level amplitudes in $\mathcal{N}=4$ SYM theory, is a projection of the positive Grassmannian $\mathrm{Gr}_{\geq 0}(k,n)$ to $\mathrm{Gr}(k,k+m)$ under a map induced by an $n\times(k+m)$ matrix with positive maximal minors. One of the main conjectures in \cite{arkani2017positive} is that (tree) amplituhedra are positive geometries. This conjecture was resolved in \cite{ranestad2024adjoints} for the case $k=m=2$ by studying the adjoints of amplituhedra living in $\mathrm{Gr}(2,4)$. 

In this paper, we continue the quest for positive geometries and study adjoints of another family of subsets of $\mathrm{Gr}(2,4)$. These subsets are obtained from the positive Grassmannian $\mathrm{Gr}_{\geq 0}(2,4)$ by adding one or two additional linear inequalities. We call the resulting sets \emph{positive pentahedra} and \emph{hexahedra} in $\mathrm{Gr}(2,4)$ respectively. The setup of this paper is in a certain sense dual to that of amplituhedra: instead of projecting the positive Grassmannian, we intersect it with half-spaces.

We note that our approach heavily relies on computational algebra and that the conclusions we make in this paper are supported by computations in computer algebra systems, which are made available at \cite{mathrepo}.

Our main results are as follows. When $\mathrm{Gr}_{\geq 0}(2,4)$ is cut by one additional hyperplane with some mild additional assumptions, we prove that the adjoint is unique and that positive pentahedra in $\mathrm{Gr}(2,4)$ are positive geometries. For the case of two additional hyperplanes we show that the adjoint is not necessarily unique and that the family of adjoints is at most six-dimensional. We conjecture that this family is in fact at most three-dimensional and show on an example that this bound is attained. 

The article is organized as follows. In Section \ref{sec:2} we introduce the necessary notions. 
In Section \ref{sec:3} we define the objects of our study: positive pentahedra and hexahedra in $\mathrm{Gr}(2,4)$. Section \ref{sec:4} features our main results on adjoints. Section \ref{sec:5} is devoted to examples: we present three positive hexahedra exhibiting different dimensions of the family of adjoints, namely one and three. In all of these examples, we show that the considered hexahedra are positive geometries. Finally, in Section \ref{sec:6} we discuss open questions and directions for future work. 

\section{Preliminaries} \label{sec:2}
In this section we give the preliminaries that are necessary to set up our problem. We start from the notions of residual arrangements and adjoint hypersurfaces of polytopes in the (complex) projective space $\mathbb{P}^n$.  By a polytope $P\subset \mathbb{P}^n $ we mean a convex full-dimensional polytope inside the real points of some affine chart of $\mathbb{P}^n$. We denote the hyperplane arrangement in $\mathbb{P}^n$ obtained by taking Zariski closures of the facets of $P$ by $\mathcal{H}_P$. 

\begin{definition}[Residual arrangement of a polytope] \label{def:ra_pol}
    Let $P\subseteq \mathbb{P}^n$ be a polytope. The collection of all intersections of hyperplanes in $\mathcal{H}_P$ that do not contain a face of $P$ is called the \emph{residual arrangement} of $P$ and is denoted by $\mathcal{R}(P)$. 
\end{definition}

\begin{definition}[Adjoint hypersurface/polynomial of a polytope]
Let $P\subseteq \mathbb{P}^n$ be a polytope with $d$ facet hyperplanes. A hypersurface in $\mathbb{P}^n$ of degree $d-n-1$ containing $\mathcal{R}(P)$ is called an \emph{adjoint hypersurface} of $P$. The defining polynomial of an adjoint hypersurface is called an \emph{adjoint polynomial}.
\end{definition}
Adjoint hypersurfaces are classical objects in algebraic geometry. After blowing up the residual arrangement, their strict transform restricts to canonical divisors on the boundary.
\begin{example} \label{ex:intro}
    For a pentagon in the plane, the residual arrangement consists of five points. These are marked in red in Figure \ref{fig:pentagon}. There is a unique conic interpolating these five points, shown in orange. This is the unique adjoint of the pentagon. 
\end{example}

\begin{figure}
            \centering
            \begin{tikzpicture}[scale=1, line width=0.5mm]
            \draw[black] (-1,0) -- (4,0);
            \draw[black] (0,-1) -- (0,4);
            \draw[black] (-2,3) -- (3,-2);
            \draw[black] (-2,2.5) -- (4,-0.5);
            \draw[black] (-0.5,4) -- (2.5,-2);

            \begin{scope}[shift={(1,1)},rotate=-45]
            \draw[orange] (0,0) ellipse (2.45cm and 1.42cm);
            \end{scope}

            \draw node[red,fill,circle,minimum size=0.2cm,inner sep=0pt] at (0,0) {};
            \draw node[red,fill,circle,minimum size=0.2cm,inner sep=0pt] at (0,3) {};
            \draw node[red,fill,circle,minimum size=0.2cm,inner sep=0pt] at (-1,2) {};
            \draw node[red,fill,circle,minimum size=0.2cm,inner sep=0pt] at (2,-1) {};
            \draw node[red,fill,circle,minimum size=0.2cm,inner sep=0pt] at (0,0) {};
            \draw node[red,fill,circle,minimum size=0.2cm,inner sep=0pt] at (3,0) {};

            \fill[blue, opacity=0.3] (0,1) -- (0,1.5) -- (1,1) -- (1.5,0) -- (1,0) -- cycle;

        \end{tikzpicture}
            \caption{Residual arrangement (red) and adjoint curve (orange) of a pentagon in the plane.}
            \label{fig:pentagon}
\end{figure}

Adjoint hypersurfaces play an important role in the context of positive geometries \cite{arkani2017positive}, where their defining polynomials arise as numerators of canonical forms of polytopes. We now define positive geometries and canonical forms, following \cite{Lam:2022yly}. Let $X$ be a complex $d$-dimensional irreducible algebraic variety defined over $\mathbb{R}$. We equip the real points $X(\mathbb{R})$ with the analytic topology. Let $X_{\geq 0} \subset X(\mathbb{R})$ be a closed semialgebraic subset such that the interior~$X_{>0} = \operatorname{Int}(X_{\geq 0})$ is an oriented $d$-manifold, and the closure of~$X_{>0}$ recovers~$X_{\geq 0}$. Let~$\partial X_{\geq 0}$ denote the boundary~$X_{\geq 0} \setminus X_{>0}$ and let~$\partial X$ denote the \emph{algebraic boundary} of $X$, i.e.\ the Zariski closure of~$\partial X_{\geq 0}$. Let~$C_1, C_2,\ldots, C_r$ be the irreducible components of~$\partial X$. Let~$C_{i,\geq 0}$ denote the closures of the interior of~$C_i \cap X_{\geq 0}$ in $C_i(\mathbb{R})$.
The spaces~$C_{1,\geq 0},\ldots, C_{r,\geq 0}$ are called the boundary components, or \emph{facets} of~$X_{\geq 0}$. In Figure \ref{fig:pentagon} the algebraic boundary of the pentagon consists of the five black lines. The components $C_i$ are the individual lines, and the sets $C_{i,\geq 0}$ are the edges of the pentagon. 

\begin{definition}[Positive geometries] \label{def:posgeo}
    A pair~$(X,X_{\geq 0})$ is a \emph{positive geometry} if there exists a unique nonzero rational $d$-form~$\Omega (X,X_{\geq 0})$ on $X$, called the \emph{canonical form}, satisfying the following recursive axioms:
    \begin{enumerate}
        \item If $d = 0$, then $X = X_{\geq 0} = \mathrm{pt}$ is a point and we define~$\Omega(X,X_{\geq 0}) = \pm 1$ depending on the orientation.
        \item If~$d > 0$, then we require that~$\Omega (X, X_{\geq 0})$ has poles only along the boundary components~$C_i$, these poles are simple, and for each~$i = 1, 2,\ldots, r,$ the pair $(C_i, C_{i,\geq 0})$ is a positive geometry with canonical form 
        $$\Omega(C_i, C_{i,\geq 0}):=\operatorname{Res}_{C_i} \Omega(X, X_{\geq 0}).$$ Here $\mathrm{Res}_{C_i}\Omega$ denotes the (Poincar\'e) residue of $\Omega$ at $C_i$. 
    \end{enumerate}
\end{definition}

Polytopes in $\mathbb{P}^n$ are known to be positive geometries \cite[Section 5]{arkani2017positive}, and the canonical form of a polytope is given by $f(\mathbf{x})\mu_{\mathbb{P}^n}(\mathbf{x})$, where $f(\mathbf{x})$ is a rational function, $\mu_{\mathbb{P}^n}(\mathbf{x}) = \sum_{i=0}^n x_i dx_0\wedge\ldots\wedge \widehat{d x_i} \wedge\ldots\wedge dx_n$ is the standard rational top-form on $\mathbb{P}^n$ and $\mathbf{x} = (x_0:\ldots:x_n)$ is the vector of homogeneous coordinates. The numerator of $f$ is then given by the adjoint polynomial, and the denominator is the product of linear forms defining the facets \cite[Theorem 5]{Lam:2022yly}.

The uniqueness of the adjoint in Example \ref{ex:intro} is no accident. For simple polytopes in $\mathbb{P}^n$ (in the sense that the facet hyperplane arrangement of the polytope is simple) it is known \cite[Theorem 1]{kohn2020projective} that the adjoint hypersurface is unique.  

The notion of the residual arrangement (and therefore of adjoint hypersurfaces) relies exclusively on the fact that the notion of a face of a polytope is well-defined. It can therefore be generalized to nonlinear semialgebraic sets with well-behaved boundary. A notable example here is given by \emph{polypols}. Their adjoints were extensively studied in \cite{kohn2021adjoints}.

Another example of a semialgebraic set with a well-defined face structure is that of the \emph{positive Grassmannian} \cite{postnikov}. It is in fact a regular CW-complex (\cite[Theorem 3.5]{postnikov} and \cite[Theorem 1.1]{galashin2022regularity}) whose cells are indexed by positroids and a number of other combinatorial objects. We will now define this object and review its boundary stratification. 

The real Grassmannian $\mathrm{Gr}(k,n)$ is the algebraic variety parametrizing $k$-dimensional linear subspaces of $\mathbb{R}^n$ (or equivalently, $(k-1)$-dimensional subspaces of $\mathbb{P}^{n-1}$). The Grassmannian $\mathrm{Gr}(k,n)$ can be embedded into the projective space $\mathbb{P}^{\binom{n}{k}-1}$ via the Pl\"ucker embedding. This is done by sending a $k\times n$ matrix whose rows span a vector space to the vector of its maximal minors, called the \emph{Pl\"ucker coordinates}. See e.g. \cite[Chapter 5]{michalek2021invitation} for an introductory reference. 

\begin{definition}[Positive Grassmannian]
   The \emph{positive (or totally nonnegative) Grassmannian} $\mathrm{Gr}_{\geq 0}(k,n)$ is a subset of $\mathrm{Gr}(k,n)$ consisting of the points all of whose nonzero Pl\"ucker coordinates have the same sign.  
\end{definition}
The boundary strata of $\mathrm{Gr}_{\geq 0}(k,n)$ are called (non-top-dimensional) \emph{positroid cells}. They are defined by intersecting $\mathrm{Gr}_{\geq 0}(k,n)$ with linear spaces given by the vanishing of some of the Pl\"ucker coordinates. The unique top-dimensional positroid cell, defined by requiring all Pl\"ucker coordinates to be nonzero, is the interior of $\mathrm{Gr}_{\geq 0}(k,n)$. Zariski closures of positroid cells are called \emph{positroid varieties}. 

It is the stratification into positroid cells that defines the structure of a regular CW-complex on $\mathrm{Gr}_{\geq 0}(k,n)$. In this context, positroid cells are analogs of faces of a polytope. This allows one to define the residual arrangement of $\mathrm{Gr}_{\geq 0}(k,n)$ as the collection of irreducible components of intersections of positroid hypersurfaces in $\mathrm{Gr}(k,n)$ that are not contained in $\mathrm{Gr}_{\geq 0}(k,n)$ itself. This definition is unfortunately not particularly meaningful: it is known that $\mathrm{Gr}_{\geq 0}(k,n)$ is a \emph{simplex-like positive geometry} \cite[Section 5.5.2]{arkani2017positive}, that is, its residual arrangement is empty (in analogy with that of the simplex in $\mathbb{P}^n$). This similarity between the positive Grassmannian and simplices stems from the fact that $\mathrm{Gr}_{\geq 0}(k,n)$ is the intersection of the whole Grassmannian $\mathrm{Gr}(k,n)$ with the positive simplex $\mathbb{P}^{\binom{n}{k}-1}_{\geq 0}$. However, when $\mathrm{Gr}(k,n)$ is intersected with more complicated polytopes in $\mathbb{P}^{\binom{n}{k}-1}_{\geq 0}$, the resulting object may have non-empty residual arrangement. This paper studies a special case of this situation. Natural questions to ask are whether there is a unique adjoint hypersurface, and if not, what the dimension of the family of adjoint polynomials is. 

Our focus is on $\mathrm{Gr}_{\geq 0}(2,4) \subset \mathbb{P}^5$. Its positroid stratification is nicely visualized in \cite[Appendix A]{williams2005combinatorial}, and the facet hyperplanes are $p_{12} = 0$, $p_{14}=0$, $p_{23}=0$, $p_{34} = 0$. We will cut the positive simplex $\mathbb{P}^5_{\geq 0}$ by one or two additional hyperplanes, therefore adding one or two facets, and then study the intersection of the resulting polytope with $\mathrm{Gr}(2,4)$. This interection is naturally a subset of $\mathrm{Gr}_{\geq 0}(2,4)$. The precise setup is explained in Section \ref{sec:3}.

\section{Positive polytopes in $\mathrm{Gr}(2,4)$} \label{sec:3}

We adopt the notation $\mathbf{p}=(p_{12},p_{13}, p_{23}, p_{14}, p_{24}, p_{34})$ . The positive simplex $\mathbb{P}^5_{\geq 0}$ in $\mathbb{P}^5$ is defined by the inequalities $p_{ij}\geq 0$ for $1\leq i < j \leq 4$. Note that each of these equalities does not make sense individually, since the Pl\"ucker coordinates are only defined up to a common scalar. What we mean by these six linear inequalities is that the quadratic inequalities $p_{ij}p_{kl}\geq 0$ hold for any pair of Pl\"ucker coordinates. In a similar way, in what follows, whenever we define a set by linear inequalities $l_i(\mathbf{p})\geq 0$ in the projective space, we will mean the set defined by the corresponding quadratic inequalities obtained by multiplying pairs of linear forms $l_i(\mathbf{p})$ in all possible ways. 

Let $P \subset \mathbb{P}^5$ be a polytope. \emph{Facets} of the semialgebraic set $P\cap \mathrm{Gr}(2,4)$ are understood in the sense of the discussion before Definition \ref{def:posgeo} and lower-dimensional faces are defined recursively. 
The positive simplex $\mathbb{P}^5_{\geq 0}$ has six facets, given by the vanishing of the Pl\"ucker coordinates. However, only four of those yield facets of the positive Grassmannian $\mathrm{Gr}_{\geq 0}(2,4)$: the facets of the simplex defined by $p_{13} = 0$ and $p_{24} = 0$ intersect the positive Grassmannian in a locus of codimension two. This motivates the following definition. 

\begin{definition}[Positive polytopes]
    Let $P\subset \mathbb{P}^5$ be a simple polytope with $m+6$ facets such that six of the facets are contained in the Pl\"ucker hyperplanes and the remaining facets are in the hyperplanes defined by the vanishing of the linear forms $h_i(\mathbf{p}) = 0$, $i=1,\ldots,m$. Assuming that the semialgebraic set $S = P\cap \mathrm{Gr}(2,4)$ has $m+4$ facets that are given by the four positroid hyperplanes $p_{12}=0$, $p_{14}=0$, $p_{23}=0$, $p_{34}=0$ and the $m$ additional hyperplanes $h_i(\mathbf{p}) = 0$, we call $S$ a \emph{positive polytope} in $\mathrm{Gr}(2,4)$ and $P$ the \emph{associated polytope} of $S$. 
\end{definition}
Any positive polytope is contained in the positive Grassmannian, which itself a positive polytope. Note that this definition can be straightforwardly generalized to arbitrary Grassmannians $\mathrm{Gr}(k,n)$.

\begin{remark}
    The positive Grassmannian is homeomorphic to a ball \cite{galashin2022totally}. In particular, it is orientable. Intersecting it with half-spaces preserves orientability, so positive polytopes are orientable as well.
\end{remark}

The focus of this paper is on positive pentahedra and hexahedra, i.e. on positive polytopes that are defined by one or two additional hyperplanes $h_i(\mathbf{p})=0$.

\begin{definition}[Residual arrangement of a positive polytope]
    Let $S$ be a positive polytope in $\mathrm{Gr}(2,4)$. Just like in the case of polytopes in $\mathbb{P}^5$, the \emph{residual arrangement} of $S$ consists of those irreducible components of intersections of Zariski closures of facets of $S$ that do not contain a face of $S$. 
\end{definition}



\subsection{Combinatorial classification of positive hexahedra} \label{sec:31}
Intersecting a polytope with the Grassmannian yields interesting combinatorics. In this subsection we illustrate how the situation becomes more complicated compared to the case of projective polytopes on the example of positive hexahedra. 

The associated polytope of a positive hexahedron in $\mathrm{Gr}(2,4)$ is a simple five-dimensional polytope with eight facets. There are eight combinatorial types of such polytopes (they are labeled by Gale diagrams of their polar duals in \cite[Figure 6.3.3]{grunbaum1967convex}). The combinatorial type of the associated polytope determines the number of linear spaces of each dimension in its residual arrangement, as well as the incidence relations between these spaces. However, when intersected with $\mathrm{Gr}(2,4)$, two combinatorially equivalent associated polytopes can produce positive hexahedra with meaningfully different residual arrangements (see Example \ref{ex:combcomp}). For this reason, we need a finer notion of combinatorial type for positive hexahedra that would capture the presence of $\mathrm{Gr}(2,4)$. We now introduce this. 

Given a polytope in $\mathbb{P}^5$ with eight facets, one can choose coordinates on $\mathbb{P}^5$ such that any given ordered subset of six facets is given by the vanishing of the Pl\"ucker coordinates. Since the Grassmannian is given by the quadratic equation $p_{12}p_{34}-p_{13}p_{24}+p_{14}p_{23} = 0$, to determine the residual arrangement of a positive hexahedron $S$, one needs to know the combinatorial type of the associated polytope $P$ and which intersections of two facet hyperplanes are given by the vanishings of the Pl\"ucker coordinates $p_{12}=p_{34}=0$, $p_{13}=p_{24}=0$ and $p_{14}=p_{23}=0$. We now elaborate on this. Suppose an octahedron $P$ in $\mathbb{P}^5$ is defined by the facet inequalities $h_i(\mathbf{p})\geq0,\ i=1,\ldots 8$. Let $I\subseteq [8]$. To determine whether the intersection of facet hyperplanes $\{h_i(\mathbf{p})=0\ :\ i\in I\}$ is in the residual arrangement $\mathcal{R}(P)$, one needs to determine whether the system of equations and inequalities $h_i(\mathbf{p})=0, \ i\in I$ and $h_j(\mathbf{p})\geq 0, \ j\not\in I$ has a solution. Intersecting the polytope $P$ with a quadratic hypersurface $q(\mathbf{p}):=h_{i_1}(\mathbf{p})h_{i_2}(\mathbf{p})-h_{i_3}(\mathbf{p})h_{i_4}(\mathbf{p})+h_{i_5}(\mathbf{p})h_{i_6}(\mathbf{p})=0$ yields a semialgebraic set $S$. To determine whether the intersection of facets of this set $\{q(\mathbf{p})=0,\ h_i(\mathbf{p})=0\ :\ i\in I\}$ is in the residual arrangement $\mathcal{R}(S)$, one now needs to know which of the $h_i$'s participate in the equation for $q$. Note that permuting $i_{1}$ and $i_2$ (and analogously $i_3$ and $i_4$ as well as $i_5$ and $i_{6}$) does not change the equation of $q$. Moreover, simultaneously exchanging $i_1$ and $i_5$ and $i_2$ and $i_6$ does not change this equation. In our setup, the quadric $q$ is the Pl\"ucker quadric, and we need to choose which of the $h_i$'s define its three monomials. 

We say that two positive hexahedra $S_1$ and $S_2$ with associated polytopes $P_1$ and $P_2$ defined by additional hyperplanes $h_1(\mathbf{p}), h_2(\mathbf{p})$ and $h_1'(\mathbf{p}), h_2'(\mathbf{p})$ respectively are \emph{combinatorially equivalent} if $P_2$ can be obtained from $P_1$ by replacing $h_1$ and $h_2$ with $h_1'$ and $h_2'$ (the order is not important) and permuting the Pl\"ucker coordinates in a way that preserves the Pl\"ucker quadric $p_{12}p_{34}-p_{13}p_{24}+p_{14}p_{23}$.

By the definition above, a choice of combinatorial type of $S$ is a choice of combinatorial type of $P$ along with the choice of three intersections of hyperplanes defined by the equations $p_{12}=p_{34}=0$, $p_{13}=p_{24}=0$ and $p_{14}=p_{23}=0$, in which we identify the choices that exchange the spaces $p_{12}=p_{34}=0$ and $p_{14}=p_{23}=0$. There are $420$ ways to choose three intersections of two out of eight facet hyperplanes, which come in pairs whose elements get identified as above. For each combinatorial type of the associated polytope this yields $210$ combinatorial types of positive hexahedra in $\mathrm{Gr}(2,4)$. The total number of combinatorial types of positive hexaherda is therefore $8\cdot 210 = 1680$.

\begin{example}\label{ex:combcomp}
    Consider the positive hexahedron $S$ defined by the additional hyperplanes 
   \begin{gather*}
         h_1(\mathbf{p}) = p_{12} - 3p_{13} - p_{23} + 8p_{14} + 2p_{24}+2p_{34},\\
    h_2(\mathbf{p}) = 2p_{12} + p_{13} - 5p_{23}+p_{14}-3p_{24}+p_{34}
    \end{gather*}
and the positive hexahedron $S'$ with additional hyperplanes 
\begin{gather*}
     h_1(\mathbf{p}) = p_{12} - 3p_{13} + p_{23} - 8p_{14} + 2p_{24}+2p_{34} ,\\
    h_2(\mathbf{p}) = 2p_{12} + p_{13} - 5p_{23}+p_{14}+3p_{24}-p_{34}.
\end{gather*}
The associated polytopes of $S$ and $S'$ are combinatorially equivalent: the residual arrangement for both of them consists of a $\mathbb{P}^3$ and two planes. The residual arrangements of $S$ and $S'$, however, are different. The residual arrangement of $S$ consists of two planes, four lines and two points, while that of $S'$ consists of eight lines, a conic and a point. We will investigate the residual arrangement of $S$ in more detail in Example \ref{ex:schub}.
\end{example}

\section{Adjoints} \label{sec:4}
In this section we investigate adjoint hypersurfaces of positive pentahedra and hexahedra in $\mathrm{Gr}(2,4)$. We start with a definition. 

\begin{definition}[Adjoints of positive polytopes in $\mathrm{Gr}(2,4)$]
    Let $S$ be a positive polytope in $\mathrm{Gr}(2,4)$ with $m+4$ facets. A hypersurface $A$ of degree $m$ in $\mathbb{P}^5$ interpolating the residual arrangement of $S$ such that $A\cap \mathrm{Gr}(2,4)$ has codimension one in $\mathrm{Gr}(2,4)$ is called \emph{adjoint}.
\end{definition}

The choice of degree in the definition above is explained as follows. Recall that the canonical bundle of $\mathbb{P}^n$ is $\mathcal{O}(-n-1)$ and that of $\mathrm{Gr}(2,4)$ is $\mathcal{O}(-4)$. Thus, any global rational differential top-form on $\mathbb{P}^n$ has degree $-n-1$, and any such form on $\mathrm{Gr}(2,4)$ has degree $-4$. For a polytope in $\mathbb{P}^n$ or a positive pentahedron/hexahedron in $\mathrm{Gr}(2,4)$ such a form is given by the ratio of the adjoint polynomial and the product of the facet equations. This means that for a polytope with $8$ facets in $\mathbb{P}^5$ the adjoint hypersurface is quadratic, and for a polytope with $7$ facets it is linear, i.e.\ a hyperplane. Analogously, a positive pentahedron in $\mathrm{Gr}(2,4)$ has a linear adjoint, and a positive hexahedron has a quadratic one.

\subsection{Pentahedra}
We now consider the case of positive pentahedra. Let $S$ be a positive pentahedron in $\mathrm{Gr}(2,4)$ and $P$ its associated polytope in $\mathbb{P}^5$ with $h(\mathbf{p}) = 0$ being the additional (non-coordinate) facet of $P$. Suppose 
$$h(\mathbf{p}) = \sum\limits_{I \in \binom{[4]}{2}} c_Ip_I,$$
where $c_I$ are real numbers. Note that $P$ is simple if and only if $h(\mathbf{p})$ does not vanish at the coordinate points in $\mathbb{P}^5$, that is, if $c_I\neq 0$ for all $I$.

\begin{theorem} \label{thm:penta}
    Suppose $c_I\neq 0$ for all $I$, and suppose the hyperplane $H$ defined by $h(\mathbf{p}) = 0$ intersects every facet of $\mathrm{Gr}_{\geq 0}(2,4)$. Then there exists a unique adjoint of the positive pentahedron $S = P \cap \mathrm{Gr}(2,4)$, and $S$ is a positive geometry. The numerator of its canonical form (and the defining polynomial of the adjoint) is $\sum \overline{c_I}p_I$, where $\overline{c_I} = \max (c_I, 0)$.
\end{theorem}

\begin{proof}
    Let $S_1$ and $S_2$ be the pieces into which $\mathrm{Gr}_{\geq 0}(2,4)$ is cut by $H$. In the notation above, we have $S_1 = S$ and $S_2 = (\mathrm{Gr}_{\geq 0}(2,4) \setminus S) \cup (\mathrm{Gr}_{\geq 0}(2,4) \cap H)$. We will now describe the residual arrangements $\mathcal{R}_1$ and $\mathcal{R}_2$ of $S_1$ and $S_2$, respectively. By the first assumption, each facet of $\mathrm{Gr}_{\geq 0}(2,4)$ contributes a facet to both $S_1$ and $S_2$. The residual arrangement $\mathcal{R}_i$ consists of two parts. The first one (we denote it $\mathcal{R}^G_i$) consists of the intersections of facet hyperplanes of $\mathrm{Gr}_{\geq 0}(2,4)$ that do not give a face of $S_i$. The second one (we denote in $\mathcal{R}^H_i$) consists of the intersections of $H$ with some of the facet hyperplanes of $\mathrm{Gr}_{\geq 0}(2,4)$ that do not give a face of $S_i$. Since $c_I\neq 0$ for all $I$ by assumption, $H$ does not contain any vertex of $\mathrm{Gr}_{\geq 0}(2,4)$. This means that each vertex of $\mathrm{Gr}_{\geq 0}(2,4)$ is either in $\mathcal{R}^G_{1}$ or in $\mathcal{R}^G_2$. More precisely, $v_I \in \mathcal{R}^G_{1}$ if $c_I <0$ and $v_I \in \mathcal{R}^G_2$ if $c_I > 0$. Note that if for two vertices $v$ and $w$ connected by an edge of $\mathrm{Gr}_{\geq 0}(2,4)$ we have $v \in \mathcal{R}^G_{1}$ and $w \in \mathcal{R}^G_2$, then the edge of $\mathrm{Gr}_{\geq 0}(2,4)$ between them contributes an edge to both $S_1$ and $S_2$, and is in none of the residual arrangements. This means that the span of $\mathcal{R}^G_i$ is equal to the span of the vertices of $\mathrm{Gr}_{\geq 0}(k,n)$ that are in $\mathcal{R}^G_i$. 

We now turn to characterizing $\mathcal{R}^H_i$. If $F$ is a face of $\mathrm{Gr}_{\geq 0}(2,4)$ that contributes a face to both $S_1$ and $S_2$, then it has to intersect $H$ inside $\mathrm{Gr}_{\geq 0}(2,4)$, and therefore $F\cap H$ does not contribute to either of the residual arrangements. If $F$ only contributes a face to $S_1$, then it is in $\mathcal{R}^G_2$ and the hyperplane spanned by $F$ intersects $H$ outside of $\mathrm{Gr}_{\geq 0}(2,4)$. This intersection is therefore in $\mathcal{R}^H_1$. If $F$ only contributes a face to $S_2$, then it already is in $\mathcal{R}^G_1$ and only contributes a lower-dimensional piece to $\mathcal{R}^H_1$. Since each face $F$ of $\mathrm{Gr}_{\geq 0}(2,4)$ contributes a face to etiher $S_1$ or $S_2$, and since all the irreducible components of $\mathcal{R}_i^H$ are by definition Zariski closures of the sets of the form $F \cap H$ for some face $F$ of $\mathrm{Gr}_{\geq 0}(2,4)$, one sees that $\mathrm{span}(\mathcal{R}^H_1) = \mathrm{span}(H\cap \mathcal{R}^G_2)$ and, symmetrically, $\mathrm{span}(\mathcal{R}^H_2) = \mathrm{span}(H\cap \mathcal{R}^G_1)$. Since $\mathcal{R}_i = \mathcal{R}_i^G \cup \mathcal{R}_i^H$, we now can describe the span of $\mathcal{R}_i$, which is a projective space.

Since the span of $\mathcal{R}_2^G$ is equal to the span of its vertices, the span of $H \cap \mathcal{R}_2^G$ is equal to the span of the intersection points $w_j$ of the edges between these vertices with $H$.
Any hyperplane interpolating $\mathcal{R}_1^G$ has an equation of the form 
$$\sum\limits_{I \in \binom{[n]}{k}} b_I p_I,$$
where $b_I =0$ if $v_I \in \mathcal{R}_1$. The fact that the adjoint hyperplane should also interpolate the points $w_j$ imposes the conditions $b_I = Ac_I$ if $v_I\in S_1$ for the adjoint, where $A$ is some scalar factor. This is because each point $w_j$ is of the form $\alpha v_I + \beta v_J$ for the vertices $v_I$ and $v_J$ in $\mathcal{R}_G^2$. In addition, for any hyperplane interpolating $w_j$ we have $0=\alpha b_I+ \beta b_J$, and, since $H$ contains $w_j$, we also have $0=\alpha c_I+ \beta c_J$ for all $I$ and $J$ such that $v_I, v_J \in \mathcal{R}_G^2$. This determines the equation of the adjoint $a(\mathbf{p})$ uniquely up to a scalar factor. This scalar factor is then fixed by computing the residue of the resulting canonical form at any vertex of $S_1$ and requiring it to be $\pm 1$ depending on the chosen orientation.

The conditions $b_I=c_I$ for $v_I\in S_1$ can alternatively be obtained by imposing the residue conditions for the canonical form at the vertices $v_I \in S_1$. This shows that the form 
$$
\Omega(\mathrm{Gr}(2,4), S_1) = \dfrac{a(\mathbf{p}) dp_{12}\wedge dp_{23} \wedge dp_{34} \wedge dp_{14} }{p_{12}p_{23}p_{34}p_{14} h(\mathbf{p})}.
$$
is the canonical form of the positive pentahedron $S_1$ and that $S_1$ is a positive geometry. Analogously one can show that $S_2$ is also a positive geometry.

\end{proof}


\begin{remark}
    Theorem \ref{thm:penta} generalizes straightforwardly to arbitrary Grassmannians $\mathrm{Gr}(k,n)$. We are, however, unaware how restrictive the condition that $H$ intersects all the facets of $\mathrm{Gr}_{\geq 0}(k,n)$ is for higher $k$ and $n$. For $\mathrm{Gr}(2,4)$ this is a requirement that is not difficult to fulfill, as shown in the next example. 
\end{remark}

\begin{example}
    We fix the order of Pl\"ucker variables to be $p_{12},p_{13},p_{23},p_{14},p_{24},p_{34}$. Consider the hyperplane defined by the linear form
    $$h(\mathbf{p}) = p_{12} - 3p_{13} + p_{23} - 8p_{14} + 2p_{24} + 2p_{34}.$$
We consider the positive pentahedron $S$ in $\mathrm{Gr}(2,4)$ defined by $p_{ij}\geq 0$ and $h(\mathbf{p})\geq 0$. Its facets are given by the four facets of $\mathrm{Gr}_{\geq 0}(2,4)$, namely $\{p_{12} = 0\}$, $\{p_{14} = 0\}$, $\{p_{23} = 0\}$ and $\{p_{34} = 0\}$, together with $\{h(\mathbf{p}) = 0\}$. The hyperplane defined by $h(\mathbf{p})$ intersects all four facets of $\mathrm{Gr}_{\geq 0}(2,4)$. In the notation of Theorem \ref{thm:penta}, we have that $\mathcal{R}^G$ consists of the two vertices $(0:1:0:0:0:0)$ and $(0:0:0:1:0:0)$ and the line containing them. Any hyperplane containing this line has an equation of the form $b_{12}p_{12}+b_{23}p_{23}+b_{24}p_{24}+b_{34}p_{34}=0$. Now, the remaining four coordinate points are connected by $6$ lines in $\mathbb{P}^5$ but only five of these lines are entirely in $\mathrm{Gr}(2,4)$ (the line through $(1:0:0:0:0:0)$ and $(0:0:0:0:0:1)$ only intersects the Grassmannian in these two points). The five lines intersect the hyperplane defined by $h(\mathbf{p})$ in five points. These points are in the residual arrangement of $S$, and the unique hyperplane containing these five points and $\mathcal{R}^G$ is defined by the equation $p_{12} + p_{23} + 2p_{24} + 2p_{34} = 0$. This is precisely the adjoint hyperplane of $S$.  
\end{example}

\subsection{Hexahedra}
Consider two hyperplanes in the Pl\"ucker space $\mathbb{P}^5$ given by the equations $h_1(\mathbf{p}) = 0$ and $h_2(\mathbf{p}) = 0$. We are interested in the region $S \subseteq \mathrm{Gr}_{\geq 0}(2,4)$ defined by the inequalities $p_{ij} \geq 0$ for any Pl\"ucker coordinate $p_{ij}$ and $h_1(\mathbf{p})\geq 0$, $h_2(\mathbf{p})\geq 0$ and a single equation $p_{12}p_{34} - p_{13}p_{24}+p_{14}p_{23}=0$ defining $\mathrm{Gr}(2,4)$. We assume that the associated polytope $\mathcal{P}_S$ has eight facets, and that $S$ has six facets. We denote the residual arrangement of $S$ by $\mathcal{R}(S)$ and the residual arrangement of $\mathcal{P}_S$ by $\mathcal{R}(\mathcal{P}_S)$.

The residual arrangements of two combinatorially equivalent positive hexahedra have the same number of components of each dimension and degree, and the same incidence relations between them. This means that for generic choices of $h_1(\mathbf{p})$ and $h_2(\mathbf{p})$ the dimension of the family of adjoints is an invariant of the combinatorial type of $S$. Since the number of combinatorial types of positive hexahedra is finite, one could in principle explicitly compute the number of adjoints for each type. However, because the number of combinatorial types is large and computing the ideal of adjoints is relatively complicated, we refrain from fully performing this computation and instead present partial results. We start by relating the residual arrangements of $S$ and $\mathcal{P}_S$.


\begin{proposition} \label{prop:posplanes}
    Let $V$ be the variety in $\mathbb{P}^5$ defined by the vanishing of the Pl\"ucker monomials: $p_{12}p_{34} = p_{13}p_{24} = p_{14}p_{23} = 0$. Assume that the facets of $S$ are in $\{h_1=0\}$, $\{h_2=0\}$ and in the positroid hyperplanes. Then $V \cap \mathcal{R}(S) = V \cap \mathcal{R}(\mathcal{P}_S)$. 
\end{proposition}

\begin{proof}
    The variety $V$ consists of the $8$ positroid planes in $\mathrm{Gr}(2,4)$, that is, of the linear two-dimensional components of the algebraic boundary of $\mathrm{Gr}_{\geq 0}(2,4)$. In particular, each point of $V$ is contained in at least two of the four positroid hyperplanes $\{p_{12}=0\}$, $\{p_{14}=0\}$, $\{p_{23}=0\}$, $\{p_{34}=0\}$. Let $X$ be a component of $\mathcal{R}(\mathcal{P}_S)$. Then, due to the observation above and because $X\cap V \subseteq V \subseteq \mathrm{Gr}(2,4)$, each component of $X\cap V$ is contained in the intersection of Zariski closures of at at least two facets of $S$. Thus, since each component of $X\cap V$ is equal to the intersection of a subset of the facet hyperplanes of $\mathcal{P}_S$, $X\cap V$ is either in $S$ or in $\mathcal{R}(S)$. Since, due to convexity, $X$ is disjoint from $\mathcal{P}_S$ and $S = \mathcal{P}_S\cap \mathrm{Gr}(2,4)$, we conclude that $X\cap V$ is disjoint from $S$ and is therefore in $\mathcal{R}(S)$. Thus, $V \cap \mathcal{R}(\mathcal{P}_S)\subseteq V\cap\mathcal{R}(S)$.  

    Conversely, let $Y$ be a component of $\mathcal{R}(S)$. Then $Y\cap V$ is contained in the intersection of facet hyperplanes of $S$. Each facet hyperplane of $S$ is also a facet hyperplane of $\mathcal{P}_S$, so $Y\cap V$ is either in $\mathcal{P}_S$ or in $\mathcal{R}(\mathcal{P}_S)$. Since $S = \mathcal{P}_S\cap \mathrm{Gr}(2,4)$ and $Y\cap V \subseteq V \subseteq \mathrm{Gr}(2,4)$, and $Y\cap V \not \subset S$, we conclude $Y\cap V \not\subset \mathcal{P}_S$ and thus $Y\cap V \subseteq \mathcal{R}(\mathcal{P}_S)$. Thus, $V\cap\mathcal{R}(\mathcal{S})\subseteq V\cap \mathcal{R}(\mathcal{P}_S)$.
\end{proof}

\begin{proposition} \label{prop:sixquadrics}
    For generic choices of the hyperplanes $h_1(\mathbf{p})$ and $h_2(\mathbf{p})$ there are at most six linearly independent adjoints of $S$ modulo the Pl\"ucker quadric defining $\mathrm{Gr}(2,4)$.
\end{proposition}

\begin{proof}
    By Proposition \ref{prop:posplanes}, the number of quadrics in $I(\mathcal{R}(S))$ is bounded from above by the number of quadrics in $I(V \cap \mathcal{R}(\mathcal{P}_S)) = \sqrt{I(V)+I(\mathcal{R}(\mathcal{P}_S))}$. We claim that as long as $h_1$ and $h_2$ are generic, the ideal $\sqrt{I(V)+I(\mathcal{R}(\mathcal{P}_S))}$ has at most seven linearly independent quadrics. The proof of this latter statement is a direct computation. The code is available at \cite{mathrepo}. Since one of these quadrics is necessarily the defining polynomial of $\mathrm{Gr}(2,4)$, the family of adjoints is at most six-dimensional. 
\end{proof}

The bound from Proposition \ref{prop:sixquadrics} is likely not tight. In all the examples we computed this family was at most three-dimensional (see Example \ref{ex:manyadj}).

\begin{conjecture}
For a generic positive hexahedron $S$ there are at most three linearly independent adjoint polynomials modulo the Pl\"ucker quadric defining $\mathrm{Gr}(2,4)$.
\end{conjecture}

\section{Examples} \label{sec:5}
In this section we present examples of positive hexahedra $S$ in $\mathrm{Gr}(2,4)$ exhibiting different dimensions of the family of adjoints. In all of these examples there is however a unique canonical form with the required residues at the boundary vertices. We obtained these examples by using computer algebra systems. We identified the residual arrangements using the \texttt{Mathematica} implementation of cylindrical algebraic decomposition (see e.g. \cite{arnon1984cylindrical}), and the equations of adjoints were obtained in \texttt{Macaulay2} \cite{M2}. We also verified that all positive hexahedra in our examples are connected and contractible, using the \texttt{Julia} package \texttt{HypersurfaceRegions.jl} presented in \cite{breiding2024computing}. The details on our code are available at \cite{mathrepo}. 

To analyze the possible canonical forms we consider one- and two-dimensional faces $F$ of the positive hexahedron $S$. They are semialgebraic sets with a boundary consisting of lower dimensional faces. We write $\overline{F}$ for the Zariski closure of $F$ and consider the arrangement ${\mathcal H}_F$ in $\overline{F}$ of Zariski closures of codimension one boundary faces in $F$. The residual arrangement $\mathcal{R}(F)$ is the collection of intersections in $\overline{F}$ of codimension at least two of Zariski closures of these boundary faces. Let ${\mathcal O}(K_F)$ be the canonical bundle on $\overline{F}$.  Similar to adjoints of $S$ in $\mathrm{Gr}(2,4)$, an adjoint hypersurface or adjoint divisor of $F$ contains the residual arrangement  $\mathcal{R}(F)$ and is the zero locus of a global section of ${\mathcal O}(K_F+{\mathcal H}_F)$. 

The intersection $\mathcal{R}(S)\cap \overline{F}$ may contain a divisor $D_F\subset \overline{F}$. The complement $(\mathcal{R}(S)\cap \overline{F})\setminus D_F$ is always contained in the residual arrangement $\mathcal{R}(F)$ of the face $F$. 
This containment may be strict.
When two disjoint faces $F_1$ and $F_2$ have the same Zariski closure, i.e. $\overline{F_1}=\overline{F_2}$, then the residual arrangement of one face may contain boundary vertices of the other, while these vertices are not contained in the residual arrangement of $S$. We will show this phenomenon in Example \ref{ex:manyadj}. 
We note that there may be adjoints to $F$ on $\overline{F}$ that are not restrictions to $\overline{F}$ of adjoints to $S$. This will also be shown in Example \ref{ex:manyadj}. 

 We argue that each example we give of a positive hexahedron is a positive geometry, in the sense that it has a unique canonical form. In some cases, the positive hexahedron  has a unique adjoint, while in other examples the adjoint is not unique, while the required $\pm 1$ residue at each boundary vertex is satisfied by a unique adjoint. 

 We use the following lemmas.
 \begin{lemma}\label{residueinvertices}
     If a positive hexahedron $S$, defined by the simple arrangement of hyperplane sections
     $\{p_{12}\cdot p_{23}\cdot p_{34}\cdot p_{14}\cdot h_1\cdot h_2=0\},$ 
     has a unique adjoint $a_S$ such that the form 
     $$\Omega_S=\frac{a_S}{p_{12}\cdot p_{23}\cdot p_{34}\cdot p_{14}\cdot h_1\cdot h_2}d\omega_S,$$ where $\omega_S$ is a volume form, 
     has iterated residues $\pm 1$ at the boundary vertices of every $1$-dimensional boundary face, then  $(\mathrm{Gr}(2,4), S)$  is a positive geometry with canonical form $\Omega_S$.   
 \end{lemma}
 \begin{proof} The form $\Omega_S$ has simple poles along the algebraic boundary of $S$, so its residue along each codimension one boundary component $F$ is a form $\Omega_F$ with iterated residues $\pm 1$ at its boundary vertices. Since the hyperplane arrangement is simple, the form $\Omega_F$ has simple poles along its boundary. The adjoint $a_S$ is unique, so the lemma follows if the form $\Omega_F$ is unique for each $F$.
 
 Now, if $\Omega_F$ is not unique, there are two forms whose poles and iterated residues at the boundary vertices coincide. Their difference is then a form $\Omega'_F$ with simple poles along the boundary of $F$ that vanishes in all boundary points.  
 But then the iterated residues of $\Omega'_F$ along one-dimensional faces also vanish. Inductively, the residue of $\Omega'_F$ along its simple codimension $1$ boundary components must vanish while it has simple poles along this boundary, so $\Omega'_F$ itself vanishes and $\Omega_F$ is unique.

We then get that the iterated residue of
 $\Omega_S$ along each face is the unique form with simple poles along the codimension one boundary of the face and iterated residues $\pm 1$ at its boundary vertices. Thus, the lemma follows. 
 \end{proof}
 
 As in \cite[Lemma 2.16]{kohn2021adjoints} there is a unique canonical form on a finite collection of disjoint intervals on the real line. The following lemma provides a criterion for a space of rational one-forms to contain this canonical form.
 \begin{lemma}\label{residuesonsegments}
     Consider an $n$-dimensional space of rational one-forms on $\mathbb{P}^1$, each with $2n-2$ zeros and with simple poles in $n$ ordered pairs of real points $(p_i,q_i), i=1..n$ so that $p_1<q_1<\ldots<p_n<q_n$ with respect to the chosen orientation. If the forms in this space have no common zeros, then there is a unique form with residue $1$ at each $p_i$ and $-1$ at each $q_i$. In particular, this form is the canonical form of the disjoint union of intervals $(p_i,q_i), i=1,\ldots,n$.
 \end{lemma}
 \begin{proof} We follow \cite[Lemma 2.16]{kohn2021adjoints}. In a parameter $t$ on $\mathbb{A}^1\subset \mathbb{P}^1$ we may assume that the points have coordinates $a_1<b_1<a_2<b_2<..<b_n$, and the $1$-forms are $$\Omega_s=\frac{f_s(t)}{(t-a_1)\cdot(b_1-t)\cdot ...\cdot (t-a_n)\cdot (b_n-t)}dt,$$ where for each $s$ in an $n$-dimensional vector space, $f_s$ is a polynomial of degree $2n-2$. The form $\Omega_s$ may be decomposed as a sum of $n$ forms by the partial fraction decomposition:
 $$\Omega_s=\sum_{i=1}^n\frac{c_{s,i}}{(t-a_i)\cdot(b_i-t)}dt.$$ This form has residues $c_{s,i}$ at $p_i$ and $-c_{s,i}$ at $q_i$, for $i=1,...,n$. 
 The forms $\Omega_s$ have no common zeros, so for a general $s$ the scalars $c_{s,i}$ are all nonzero.  
 The vector space of forms $\Omega_s$, is $n$-dimensional so there is a unique $s=s_0$, such that 
 $c_{s_0,i}=1$ for all $i$. 
 \end{proof}
\begin{remark} If the space of one-forms in the lemma has a common zero, it may not contain the canonical form.
\end{remark}
 
With these two lemmas we are ready to give examples of positive hexahedra $S$ with a unique canonical form $\Omega_S$ such that $(\mathrm Gr(2,4),S)$ is a positive geometry. 

\begin{example} \label{ex:54}
 In this example we consider an octahedron $\mathcal{P} \subset \mathbb{P}^5$ whose two additional facets are given by two hyperplanes defined over $\mathbb{Q}$:

    \begin{gather*}
            h_1(\mathbf{p}) = p_{12} - 3p_{13} - p_{23} + 8p_{14} + 2p_{24}+2p_{34} ,\\
    h_2(\mathbf{p}) = 2p_{12} + p_{13} + 5p_{23}-p_{14}-3p_{24}+p_{34}.
    \end{gather*}
    
    The residual arrangement $\mathcal{R}(S)$ of the corresponding positive hexahedron $S$ in $\mathrm{Gr}(2,4)$ is now given by six lines and a quadric surface:
    \begin{gather*}
        Q:=\{p_{12}p_{34}-p_{13}p_{24}+p_{14}p_{23} = h_1(\mathbf{p}) = h_2(\mathbf{p}) = 0\},\\
        L_1:=\{p_{12} = p_{13} = p_{23} = h_1(\mathbf{p}) = 0\},\quad L_2:=\{p_{13} = p_{23} = p_{34} = h_1(\mathbf{p}) = 0\},\\
        L_3:=\{p_{12} = p_{14} = p_{24} = h_2(\mathbf{p}) = 0\},\quad L_4:=\{p_{14} = p_{24} = p_{34} = h_2(\mathbf{p}) = 0\},\\
        L_5:=\{p_{12} = p_{13} = p_{23} = p_{34} = 0\},\quad
    L_6:=\{p_{12} = p_{14} = p_{24} = p_{34} = 0\}.
    \end{gather*}

\begin{figure}
    \centering
    \scalebox{0.7}{\begin{tikzpicture}[every node/.style={circle,draw,minimum size=8mm}]
  \node (Q) at (0,0) {$Q$};
  \node (L1) at (3,3) {$L_1$};
  \node (L2) at (3,1) {$L_2$};
  \node (L3) at (3,-1) {$L_3$};
  \node (L4) at (3,-3) {$L_4$};
  \node (L5) at (6,2) {$L_5$};
  \node (L6) at (6,-2) {$L_6$};

  \draw (Q) -- (L1);
  \draw (Q) -- (L2);
  \draw (Q) -- (L3);
  \draw (Q) -- (L4);

  \draw (L1) -- (L2);
  \draw (L1) -- (L5);
  \draw (L2) -- (L5);
  \draw (L3) -- (L4);
  \draw (L3) -- (L6);
  \draw (L4) -- (L6);
\end{tikzpicture}
}
    \caption{Intersections between the components of the residual arrangement in Example \ref{ex:54}}
    \label{fig:ex54}
\end{figure}

The incidence relations between the components of this residual arrangement are depicted in Figure \ref{fig:ex54}. There is a unique adjoint given by the polynomial $a_S=2p_{12}^2+p_{12}p_{13}+16p_{12}p_{14}+5p_{13}p_{14}+5p_{12}p_{23}+4p_{12}p_{24}+32p_{13}p_{24}+7p_{23}p_{24}-34p_{12}p_{34}+2p_{13}p_{34}+8p_{14}p_{34}+10p_{23}p_{34}+2p_{24}p_{34}+2p_{34}^2
$. 

The residual arrangement of the associated polytope $\mathcal{P}$  consists of four planes 
\begin{gather*}
\{p_{13} = p_{23} = h_1(\mathbf{p}) = 0\},\quad \{p_{14}=p_{24}=h_2(\mathbf{p})=0\},\\
\{p_{13} = h_1(\mathbf{p})=h_2(\mathbf{p})=0\},\quad \{p_{24}=h_1(\mathbf{p}) = h_2(\mathbf{p})=0\}.
\end{gather*}
and three lines
$$\{p_{12} = p_{13} = p_{23} = p_{34} = 0\}, \quad \{p_{12} = p_{14} = p_{23} = p_{34} = 0\},\quad \{p_{12} = p_{14} = p_{24} = p_{34} = 0\}.$$

The unique adjoint of $\mathcal{P}$ is given by $2p_{12}^2+p_{12}p_{13}+5p_{12}p_{23}+16p_{12}p_{14}+5p_{13}p_{14}+39p_{23}p_{14}+4p_{12}p_{24}+7p_{23}p_{24}+5p_{12}p_{34}+2p_{13}p_{34}+10p_{23}p_{34}+8p_{14}p_{34}+2p_{24}p_{34}+2p_{34}^2$. Modulo the adjoint of $S$ this is the polynomial $39p_{23}p_{14} - 32p_{13}p_{24} + 39p_{12}p_{34}$. Thus, modulo $\langle p_{12}p_{34}, p_{13}p_{24}, p_{14}p_{23}\rangle$ there is a unique adjoint of $S$ which is the same as the adjoint of $\mathcal{P}$. Note that $\mathcal{R}(\mathcal{P})$ contains all the lines of $\mathcal{R}(S)$ but does not contain the quadric surface. 

\paragraph{A positive geometry.}

Let us also see how the positive hexahedron $S$ in this  example is a four-dimensional positive geometry. 

We check the iterated residues at the boundary vertices along each one-dimensional face, and conclude by Lemma \ref{residueinvertices}. In what follows, we write $e_{ij}$ for the point in $\mathbb{P}^5$ whose vector of homogeneous coordinates satisfies $p_{ij}=1$ and $p_{kl} = 0$ for $(k,l)\neq (i,j)$. 

There are ten boundary vertices, the coordinate points $e_{12}$ and $e_{34}$ and the eight points 
$$\{h_1\cdot h_2=p_{12}=p_{13}=p_{14}=p_{34}=0\},\{h_1\cdot h_2=p_{12}=p_{24}=p_{23}=p_{34}=0\},$$
$$\{h_1=p_{14}=p_{23}=p_{24}=p_{12}p_{34}=0\},\{ h_2=p_{14}=p_{23}=p_{13}=p_{12}p_{34}=0\},$$
Each vertex is an endpoint of four one-dimensional boundary faces or boundary segments.  There are sixteen boundary line segments and four boundary segments in conics, one segment in each curve. So the algebraic closure of each boundary segment is a line or a conic that contains a pair of boundary points.
The iterated residue $\Omega_F$ of  $$\Omega_S=\frac{a_S}{p_{12}\cdot p_{23}\cdot p_{34}\cdot p_{14}\cdot h_1\cdot h_2}d\omega_S$$
along the face $F$, when $F$ is a boundary segment in a line or conic $\overline F$, has the form $$\Omega_F=\frac{a_{S,F}}{ B_1\cdots B_{r(F)}}d\omega_F,$$ where the numerator $a_{S,F}$ is the restriction of $a_S$ to $\overline F$, and $B_1,..,B_{r(F)}$ are the restrictions to $\overline F$ of those of the linear forms $p_{12},...,h_2$ that do not vanish on all of $F$.  Now the zeros of $a_{S,F}$ coincide with the zeros of the $B_i$ that are not boundary points of the segment $F$. So after reduction,   $\Omega_F$ has no zeros on $\overline F$ and  simple poles at the endpoints.  In particular, its residue has the same absolute value at the two endpoints.
Therefore we may scale $$\Omega_S=\frac{a_S}{p_{12}\cdot p_{23}\cdot p_{34}\cdot p_{14}\cdot h_1\cdot h_2}d\omega_S,$$
such that it is a canonical form for $(\mathrm{Gr}(2,4), S)$  as a positive geometry. 
\end{example}

\begin{example} \label{ex:schub}
    In this example we switch some of the signs in the forms $h_1$ and $h_2$ of the previous example. We consider an octahedron $\mathcal{P} \subset \mathbb{P}^5$ whose additional two facets are given by two Schubert hyperplanes defined over $\mathbb{Q}$ by the forms:
    \begin{gather*}
         h_1(\mathbf{p}) = p_{12} - 3p_{13} - p_{23} + 8p_{14} + 2p_{24}+2p_{34},\\
    h_2(\mathbf{p}) = 2p_{12} + p_{13} - 5p_{23}+p_{14}-3p_{24}+p_{34}.
    \end{gather*}
     As before $\mathrm{Gr}(2,4)$ is defined in Pl\"ucker coordinates by $\{qq:=p_{12}p_{34}-p_{13}p_{24}+p_{14}p_{23}=0\}$.
    The residual arrangement of the positive hexagon $S = \mathrm{Gr}(2,4)\cap \mathcal{P}$ is given by eight lines, a conic and a point. The lines and the conic are defined by the equations
    \begin{gather*}
   C_1:=\{ p_{23} = h_{1}(\mathbf{p}) = h_{2}(\mathbf{p}) = p_{12}p_{34}-p_{13}p_{24} = 0\},\\
        L_1:=\{p_{12} = p_{13} = p_{23} = h_{1}(\mathbf{p}) = 0\},\quad L_2:=\{p_{13} = p_{23} = p_{34} = h_{1}(\mathbf{p}) = 0\},\\
        L_3:=\{p_{12} = p_{13} = p_{14} = h_{1}(\mathbf{p}) = 0\}, \quad L_4:=\{p_{13} = p_{14} = p_{34} = h_{1}(\mathbf{p}) = 0\},\\
         L_5:=\{p_{23} = p_{24} = p_{34} = h_{2}(\mathbf{p}) = 0\}, \quad L_6:=\{p_{12} = p_{23} = p_{24} = h_{2}(\mathbf{p}) = 0\},\\
        L_7:=\{p_{12} = p_{14} = p_{24} = p_{34} = 0\}, \quad L_8:=\{p_{12} = p_{13} = p_{14} = p_{34} = 0\}.
    \end{gather*}
    The two points in $\{qq=p_{12}=p_{34}=h_1=h_2=0\}$, have nonnegative coordinates, except $p_{14}$ and $p_{24}$ that are connected by the equation
$11p_{14}^2-46p_{14}p_{24}-13p_{24}^2=0$. Projectively this equation has two solutions, in one of which $p_{14}$ and $p_{24}$ have different signs. So one of these points belongs to $S$, while the other one is a point in the residual arrangement. The incidence relations between the one-dimensional components of this residual arrangement are depicted in Figure \ref{fig:ex55}. The residual point is disjoint from all one-dimensional components.

\begin{figure}
    \centering
    \scalebox{0.7}{
    \begin{tikzpicture}[every node/.style={circle,draw,minimum size=8mm}]
  \node (C) at (0,0) {$C$};

  \node (L1) at (3,3) {$L_1$};
  \node (L2) at (3,1) {$L_2$};
  \node (L5) at (3,-1) {$L_5$};
  \node (L6) at (3,-3) {$L_6$};

  \node (L3) at (6,3) {$L_3$};
  \node (L4) at (6,1) {$L_4$};
  \node (L8) at (6,-1) {$L_8$};
  \node (L7) at (6,-3) {$L_7$};

  \draw (C) -- (L1);
  \draw (C) -- (L2);
  \draw (C) -- (L5);
  \draw (C) -- (L6);

  \draw (L1) -- (L2);
  \draw (L5) -- (L6);

  \draw (L1) -- (L3);
  \draw (L2) -- (L4);
  \draw (L3) -- (L4);
  \draw (L4) -- (L8);
  \draw (L8) -- (L7);
\end{tikzpicture}}
    \caption{Intersections between the components of the residual arrangement in Example \ref{ex:schub}}
    \label{fig:ex55}
\end{figure}
    
     By analyzing the intersections of  these components, one can see that, except for the isolated point,  the real part of the residual arrangement is connected. Furthermore, there is a unique adjoint $a_S$ to $S$, modulo the Pl\"ucker quadric defining $\mathrm{Gr}(2,4)$.

    The residual arrangement of the associated polytope $\mathcal{P}$ consists of a $\mathbb{P}^3$ and two $\mathbb{P}^2$'s:
    \begin{gather*}
        \{p_{13} = h_1(\mathbf{p}) = 0\},\\
    \{p_{12} = p_{14} = p_{34} = 0\},\\
    \{p_{23} = p_{24} = h_2(\mathbf{p}) = 0\}.
    \end{gather*}
    It has a unique adjoint given by $2p_{12}^2+p_{12}p_{13}+17p_{12}p_{14}+8p_{13}p_{14}+8p_{14}^2-2p_{12}p_{23}-p_{14}p_{23}+4p_{12}p_{24}+2p_{14}p_{24}+5p_{12}p_{34}+2p_{13}p_{34}+10p_{14}p_{34}-p_{23}p_{34}+2p_{24}p_{34}+2p_{34}^2$. 

   The adjoint $a_S$ to $S$ is not the restriction of this quadric to $\mathrm{Gr}(2,4)$, as is easily verified in computations. 
    We may also see this when analyzing how the residual arrangements $\mathcal{R}(S)$ and $\mathcal{R}(\mathcal{P})$ are related to each other. Each of the two $\mathbb{P}^2$'s in $\mathcal{R}(\mathcal{P})$ intersects $\mathrm{Gr}(2,4)$ in a pair of lines. These two pairs of lines appear in $\mathcal{R}(S)$. The remaining four lines of $\mathcal{R}(S)$ are contained in the $\mathbb{P}^3$ defined by $p_{13}=h_1(\mathbf{p})=0$. The two-dimensional intersection of this~$\mathbb{P}^3$ with the Grassmannian cannot appear in the residual arrangement of the region, since it is contained entirely in the ``redundant'' hyperplane $p_{13}=0$, which does not give a facet of~$S$. However, the four lines obtained by further intersecting with either $p_{12} =0$ or $p_{34}=0$ (we get a pair of lines for each of the two intersections) are exactly the four remaining lines of $\mathcal{R}(S)$.
    Now, the conic in $\mathcal{R}(S)$ does not come from $\mathcal{R}(\mathcal{P})$. The $\mathbb{P}^2$ defined by $h_1(\mathbf{p}) = h_2(\mathbf{p}) = p_{23} = 0$ contains a face of $\mathcal{P}$. However, this face is disjoint from $\mathrm{Gr}_{\geq 0}(2,4)$. So the intersection of this $\mathbb{P}^2$ with $\mathrm{Gr}(2,4)$ is entirely in the residual arrangement. This is precisely the conic $p_{13} = h_1(\mathbf{p}) = h_2(\mathbf{p}) = p_{12}p_{34} - p_{13}p_{24}= 0$.
This conic is not contained in the quadric adjoint of $\mathcal{P}$, so therefore the adjoint of $S$ is not the restriction of this quadric. 

\paragraph{A positive geometry.}
Let us see how this positive hexahedron $S$ is a positive geometry.
We list fourteen boundary vertices.  The first three are the coordinate points $$e_{12},e_{14},e_{34}.$$
Next, we find boundary vertices on two-dimensional faces of $S$ in quadric surfaces. There are three points on the quadric surface $\{qq=h_1=h_2=0\}$ of intersection between the conic plane sections $\{p_{12}p_{14}p_{34}=0\}$. There are four points on the quadric surface $\{qq=p_{12}=p_{34}=0\}$, the coordinate point $e_{14}$ and three points of intersection between the conic plane sections $\{p_{23}h_1h_2=0\}$.
Furthermore, there are six points on the quadric surface $\{qq=p_{14}=p_{23}=0\}$, the coordinate points $e_{12}, e_{34}$ and two points  on  each of the conic plane sections $\{h_1h_2=0\}$, so that the four plane sections $\{p_{12}h_1p_{34}h_2=0\}$ meet cyclically in boundary vertices. Finally two boundary vertices are the endpoints of a boundary segment on the line $\{p_{13}=p_{14}=p_{24}=h_2\}$ defined by $\{p_{12}p_{34}=0\}$.

Altogether there are seven boundary segments on conics, and $21$ boundary segments on lines, with one segment on each curve.  Each boundary vertex is a boundary point of four segments. The one-skeleton of $S$ is connected, so as in the previous example, the  
iterated residue of the form
$$\Omega_S=\frac{a_S}{p_{12}\cdot p_{23}\cdot p_{34}\cdot p_{14}\cdot h_1\cdot h_2}d\omega_S$$ along each boundary segment has no zeros and two simple poles.  By the connectedness of the one-skeleton one may scale $a_S$ so that the iterated residues of  $\Omega_S$ at the boundary vertices are $\pm 1$.  By Lemma \ref{residueinvertices}, $(\mathrm{Gr}(2,4),S)$  is a positive geometry with canonical from $\Omega_S.$
\end{example}

\begin{example} \label{ex:manyadj}
Our final example of a positive hexahedron has a three-dimensional space of adjoints, but is still a positive geometry.  In this example there are distinct two-dimensional faces with the same algebraic closure, in fact a quadric surface containing two disjoint two-dimensional faces of the hexahedron.  In this case the residual arrangement of one face contains boundary vertices of the other.  In particular, the residual arrangement of one face is not contained in the residual arrangement of the hexahedron.

The two additional facets of the positive hexahedron $S$ are defined by the linear forms
    \begin{gather*}
    h_1(\mathbf{p}) = -3p_{12} + 4p_{13}-p_{23}-p_{14}+3p_{24}-2p_{34},\\
    h_2(\mathbf{p}) = 4p_{12}-6p_{13}+ 3p_{23}+ 2p_{14}-5p_{24} + 2p_{34}.
    \end{gather*}
    inside the positive Grassmannian. 
    The residual arrangement of the associated polytope $\mathcal{P}$ consists of two $\mathbb{P}^3$'s and a line:
    \begin{gather*}
    \{p_{12} = p_{34} = h_1(\mathbf{p}) = h_2(\mathbf{p})= 0\},\\
        \{p_{13} = p_{24} = 0\},\quad \{p_{14} = p_{23} = 0\}.
    \end{gather*}
    
    The residual arrangement $\mathcal{R}(S)$ consists of a quadric surface, four lines and two points: 
    \begin{gather*}
        Q:=\{p_{14} = p_{23} = qq = 0\},\\
        L_1:=\{p_{12} = p_{13} = p_{23} = p_{24} = 0\},\\
        L_2:=\{p_{13} = p_{23} = p_{24} = p_{34} = 0\},\\
       L_3:=\{ p_{12} = p_{13} = p_{14} = p_{24} = 0\},\\ 
        L_4:=\{p_{13} = p_{14} = p_{24} = p_{34} = 0\},\\
        \{p_{12} = p_{34} = h_1(\mathbf{p}) = h_2(\mathbf{p}) = qq= 0\}.
    \end{gather*}
\begin{figure}
    \centering
  \scalebox{0.7}{\begin{tikzpicture}[every node/.style={circle,draw,minimum size=8mm}]
  \node (Q) at (0,0) {$Q$};

  \node (L1) at (-1.5,1.5) {$L_1$};
  \node (L3) at (-1.5,-1.5) {$L_3$};

  \node (L2) at (1.5,1.5) {$L_2$};
  \node (L4) at (1.5,-1.5) {$L_4$};

  \draw (Q) -- (L1);
  \draw (Q) -- (L2);
  \draw (Q) -- (L3);
  \draw (Q) -- (L4);

  \draw (L1) -- (L2);
  \draw (L2) -- (L4);
  \draw (L4) -- (L3);
  \draw (L3) -- (L1);
\end{tikzpicture}}
    \caption{Intersections between the components of the residual arrangement in Example \ref{ex:manyadj}}
    \label{fig:ex56}
\end{figure}
    
    The two points are disjoint from the higher-dimensional components, whose incidence relations are in turn depicted in Figure \ref{fig:ex56}.
    The space of quadrics interpolating the residual arrangement $\mathcal{R}(S)$ of $S$ is spanned by 
    \begin{gather*}
    2p_{13}p_{24} + 4p_{23}p_{24} + p_{14}p_{24} - 2p_{12}p_{34},\ 2p_{23}p_{14} + 4p_{23}p_{24} + p_{14}p_{24},\\ 4p_{13}p_{14} - 4p_{23}p_{24} + p_{14}p_{24},\ 12p_{13}p_{23} + 12p_{23}p_{24} + p_{14}p_{24}.
    \end{gather*}
    Modulo the Pl\"ucker quadric, the space of adjoints of $S$ is three-dimensional.

\paragraph{A positive geometry.} Since $S$ does not have a unique adjoint, to show that $(\mathrm{Gr}(2,4),S)$ is a positive geometry we need to specify an adjoint $a_S$ giving rise to the canonical form 
$$\Omega_S=\frac{a_S}{p_{12}\cdot p_{14}\cdot p_{23}\cdot p_{34}\cdot h_1\cdot h_2}d\omega_S.$$
To do this, we consider one- and two-dimensional faces of $S$.

The Zariski closures of two-dimensional faces in $S$ are planes or quadric surfaces.  These are the eight planes in $\{p_{13}p_{24} = p_{23}p_{14} = p_{12}p_{34}=0\}$ and the ten quadric surfaces 
$$\{qq=h_1=h_2=0\},\{qq=p_{12}=p_{34}=0\}$$ 
$$\{qq=h_i=p_{12}=0\},\{qq=h_i=p_{23}=0\},\{qq=h_i=p_{34}=0\},\{qq=h_i=p_{14}=0\}, i=1,2 $$
The faces of $S$ in the planes are all triangular;  each plane contains a line in the residual arrangement $\mathcal{R}(S)$.  
The quadric surfaces have faces bounded by conic sections. We consider one of these surfaces. 

The quadric surface $Q_{p_{12},p_{34}}=\{qq=p_{12}=p_{34}=0\}$ contains six residual points, and the boundary curve is defined by the four hyperplanes $h_1\cdot h_2\cdot p_{23}\cdot p_{14}=0$. So the quadratic forms in $\{p_{13},p_{14},p_{23},p_{24}\}$ that vanish in these six residual points intersect the surface $Q_{p_{12},p_{34}}$ in a net of curves that are adjoint to the face in this surface. Furthermore the restriction map is an isomorphism between the net of adjoints to $S$ and this net of curves on the quadric surface.  On the two boundary conics defined by $h_1$ and $h_2$ there are two pairs of boundary vertices.
On both $\{h_1=0\}$ and $\{h_2=0\}$ they are the intersection points of these hyperplanes with the quadrangle $\{p_{12}=p_{34}=p_{14}p_{23}=p_{13}p_{24}=0\}$.
As in \cite[Example 2.13]{kohn2021adjoints} and Lemma \ref{residuesonsegments} there is a unique (up to sign) adjoint such that the iterated residue of the form $\Omega_S$ along the boundary conic both in $\{h_1=0\}$ and in $\{h_2=0\}$ has residues $\pm 1$ at the pairs of boundary vertices.  In particular  
$$\Omega_S=\frac{a_S}{p_{12}\cdot p_{14}\cdot p_{23}\cdot p_{34}\cdot h_1\cdot h_2}d\omega_S$$ for a unique adjoint $a_S$ in the net of adjoints of $S$.

Now we need to make sure that this form $\Omega_S$ has appropriate residues at all other boundary vertices.
For this we check the set of one-dimensional boundary faces or boundary segments and their endpoints, the boundary vertices of $S$.  
The boundary segments lie on lines and conics. There are two segments in each of the six conics
$$\{p_{12}=p_{34}=h_1=0\},\{p_{12}=p_{34}=h_2=0\}, $$
$$\{p_{12}=h_1=h_2=0\},\{p_{14}=h_1=h_2=0\},
\{p_{34}=h_1=h_2=0\},\{p_{23}=h_1=h_2=0\}. $$
The latter four intersect cyclically in a pair of boundary vertices. 
 There are $20$ boundary segments in lines, each line with a single boundary segment.
 These lines are the four lines of the quadrangle $\{p_{12}=p_{34}=p_{14}p_{23}=p_{13}p_{24}=0\}$ and the $16$ lines $\{h_1h_2=p_{12}p_{34}=p_{14}p_{23}=p_{13}p_{24}=0\}$.
 There are sixteen boundary vertices, eight on the first four lines, and the eight points $\{h_1=h_2=p_{12}p_{34}=p_{14}p_{23}=p_{13}p_{24}=0\}$, such that each of them is an endpoint of four boundary segments.

 Now, we have fixed $\Omega_S$ to have iterated residues $\pm 1$ at the eight boundary points on the two conics $C_1=\{p_{12}=p_{34}=h_1=0\},C_2=\{p_{12}=p_{34}=h_2=0\}$ in $\mathrm{Gr}(2,4)$. 
 Each of these eight boundary points is an endpoint of three boundary segments in the lines. One of these segments has one of the endpoints on $C_1$ and the other on $C_2$. So two segments at each of the eight boundary points have an endpoint among the eight points in $\{h_1=h_2=0\}$ outside $\{p_{12}=p_{34}=0\}$. These eight points are contained in the eight planes $\{p_{12}p_{34}=p_{14}p_{23}=p_{13}p_{24}=0\}$, one point per plane. These planes each contain a boundary triangle with three boundary vertices. The adjoints of $S$, including the special one $a_S$ all vanish on two lines in each of these planes. These lines are also in the algebraic boundary of $S$. The iterated residue of $\Omega_S$ along the face $F_\Pi$ in such a plane $\Pi$ therefore reduces to a form $\Omega_{F_\Pi}$ with no zeros on the plane, and simple poles along the triangle boundary of $F_\Pi$. Since $\Omega_S$ has iterated residues $\pm 1$ at two boundary points of $F_\Pi$, so does $\Omega_{F_\Pi}$ which therefore has residue $\pm 1$ also in the third boundary point. So the iterated residue of $\Omega_S$ is $\pm 1$ at every boundary point of $S$, and we may conclude, by Lemma \ref{residueinvertices} that $(\mathrm{Gr}(2,4),S)$ is a positive geometry with canonical form $\Omega_S$.

Finally, we describe a new phenomenon in this example. 
In the quadric surface $Q_{h_1,h_2}=\{qq=h_1=h_2=0\}$ there are two disjoint two-dimensional faces $F_1$ and $F_2$.  The face $F_1$ has two boundary conic curves with boundary vertices $\{qq=h_1=h_2=p_{12}=p_{14}=0\}$.  The face $F_2$ is bounded by the four conics $\{qq=h_1=h_2=p_{12}\cdot p_{23}\cdot p_{34}\cdot p_{14}=0\}.$
The face $F_1$ has no residual arrangement, while $F_2$ has a residual arrangement which is the four points in 
$Q_{h_1,h_2}\cap\mathcal{R}(S)$ and the two boundary points in $F_1$ that are not on $\mathcal{R}(S)$. 

The intersection $Q_{h_1,h_2}\cap\mathcal{R}(S)$ is therefore four points. There is a five-dimensional space of quadratic forms modulo $Q_{h_1,h_2}$ that contain these four points. Among them there is the three-dimensional space forms that are restrictions of adjoints to $S$ and
three-dimensional space of adjoints to $F_2$ that vanish on all six residual points in $Q_{h_1,h_2}$ to the face $F_2$. 

In the boundary of $F_2$ each of the two conics $\{qq=h_1=h_2= p_{23}\cdot p_{34}=0\}$ contain two segments of the boundary and hence four boundary vertices. 

 The iterated residue of $\Omega_S$ along $F:=F_1\cup F_2$ has the form $$\Omega_F=\frac{a_{S,F}}{p_{12}\cdot p_{14}\cdot p_{23}\cdot p_{34}}d\omega_F,$$
 where $a_{S,F}$ is the restriction of $a_S$ to $Q_{h_1,h_2}$.
Each of the conics in the boundary of $F$ contains two boundary segments and four boundary points. There are altogether eight boundary points, one in each of the boundary planes of $S$. Notice that we already saw that the iterated residue of $\Omega_S$ at each of these points is $\pm 1$.  

Now, consider the iterated residue of $\Omega_S$ on the two connected components $F_1$ and $F_2$ of $F$.  
 The boundary of $F_1$ is the union of two conics 
 $\{qq=h_1=h_2=p_{12}\cdot p_{14}=0\}$ and two boundary points $\{h_1=h_2=p_{12}= p_{14}=p_{13}\cdot p_{24}=0\}$.
 The iterated residue of $\Omega_S$ on $F_1$ is a form $\Omega_{F_1}$ with poles along the two conics, no zeros, and iterated residues $\pm 1$ at the two boundary points.
 In particular $\Omega_{F_1}$ is of the form
  $$\Omega_{F_1}=\frac{1}{p_{12}\cdot p_{14}}d\omega_F,$$
 as a form on $Q_{h_1,h_2}$. 
 The boundary of $F_2$ is a cycle of six conic segments, and the
 residue $\Omega_{F_2}$ of $\Omega_S$ along $F_2$ has the form $$\Omega_{F_2}=\frac{a_{S,F_2}}{p_{12}\cdot p_{14}\cdot p_{23}\cdot p_{34}}d\omega_F,$$
where $a_{S,F_2}$ is an adjoint of $F_2$ on $Q_{h_1,h_2}$. It vanishes at the four points of intersection $Q_{h_1,h_2}\cap\mathcal{R}(S)$ between the residual arrangement of $S$ and $Q_{h_1,h_2}$, and at the boundary vertices of $F_1$. The iterated residues of $\Omega_{F_2}$ at the boundary vertices of $F_2$ are $\pm 1$, while the residues of this form at the boundary vertices of $F_1$ vanish.  

Let $a_{S,F}$ be the restriction of $a_S$ to $Q_{h_1,h_2}$.  Since $p_{23}p_{34}$ vanishes at the boundary points of $F_2$ and not at the boundary points of $F_1$, the form $a_{S,F}$ decomposes, for a suitable nonzero scalar $\lambda$, into 
 $$a_{S,F}=\lambda p_{23}p_{34}+a'_{S,F_2},$$
 such that $a'_{S,F_2}$ vanishes also at the boundary points of $F_1$.  Hence $a'_{S,F_2}$
 is an adjoint of $F_2$ on $Q_{h_1,h_2}$. Furthermore $$\Omega_{F_1}=\frac{\lambda \cdot p_{23}\cdot p_{34}}{p_{12}\cdot p_{14}\cdot p_{23}\cdot p_{34}}d\omega_F=\frac{\lambda}{p_{12}\cdot p_{14}}d\omega_F,$$ and $\lambda=\pm 1$ since the residues are $\pm 1$ at the boundary vertices.  Similarly, the residue $\Omega_{F_2}$ of $\Omega_S$ is $$\Omega_{F_2}=\frac{a'_{S,F_2}}{p_{12}\cdot p_{14}\cdot p_{23}\cdot p_{34}}d\omega_F,$$
 so $a'_{S,F_2}=a_{S,F_2}.$
The sum  $$\Omega_F=\Omega_{F_1}+\Omega_{F_2}=\pm\frac{1}{p_{12}\cdot p_{14}}d\omega_F+\frac{a_{S,F}}{p_{12}\cdot p_{14}\cdot p_{23}\cdot p_{34}}d\omega_F$$
becomes the iterated residue of $\Omega_S$ along $F$. It has simple poles along the boundary segments and iterated residues $\pm 1$ at the boundary points.  

This explains how $(Q,F)$, the quadric surface $Q=\{qq=h_1=h_2=0\}$ with the boundary face $F$ of $S$, is a positive geometry with form $\Omega_F$.  
\end{example}

\section{Open questions and future directions} \label{sec:6}
In this section we briefly present some open questions stemming from the study of positive polytopes in the Grassmannian. 

\begin{question}
    Are positive hexahedra in $\mathrm{Gr}(2,4)$ positive geometries? While there may be many adjoints, the residue conditions at the vertices of the hexahedron might still pick a unique one as the numerator of the canonical function. This was the case in Example \ref{ex:manyadj}.
\end{question}

\begin{question}
    What can one say about the topology of positive polytopes in $\mathrm{Gr}(2,4)$? Are they connected? Are they contractible? This is the case for the examples in Section \ref{sec:5} but does this hold in general? 
\end{question}

\begin{question}
    We compute residual arrangements of positive polytopes by using cylindrical algebraic decomposition in \texttt{Mathematica}. This is a complicated procedure which prevents us from doing the computation for all $1680$ combinatorial types of positive hexahedra. Is there more efficient software (or a more efficient algorithm) to do this? If not, can one develop such software? 
\end{question}

\section*{Conflict of interest statement}
The authors report there are no competing interests to declare.

\bibliographystyle{plain}
\bibliography{bibliography}

\bigskip
\bigskip

\noindent
\footnotesize
{\bf Authors' addresses:}

\smallskip

\noindent Dmitrii Pavlov,
TU Dresden and MPI MiS Leipzig
\hfill \url{dmitrii.pavlov@mis.mpg.de}

\noindent Kristian Ranestad,
Universitetet i Oslo
\hfill \url{ranestad@math.uio.no}

\end{document}